\documentclass[reqno]{amsart}%[oneside,reqno]{amsart}
\usepackage[greek,english]{babel}
\usepackage[iso-8859-7]{inputenc}
\usepackage{amsthm}
\usepackage{amsfonts}
\usepackage{graphicx}
\newtheorem{thm}{Theorem}

\newtheorem{lem}{Lemma}
\newtheorem{prop}{Proposition}
\newtheorem{rem}{Remark}
\begin{document}

\title[Linear non-degeneracy of the blow-up  in  phase separation of BEC{s}]{Linear non-degeneracy of the 1-D blow-up limit in the  phase segregation of  Bose-Einstein condensates}

\author{ Christos Sourdis}

\address{Institute of Applied and Computational Mathematics, Foundation of Research and Technology of Hellas, Herakleion, Crete,
 Greece}

\email{sourdis@uoc.gr}

\begin{abstract}
We show that  the kernel of the linearization of the blow-up problem at the regular part of the interface that separates segregated BECs is one-dimensional, generated by translations in the normal direction to the interface. This useful non-degeneracy property was previously known only in one and two dimensions.
\end{abstract}

%%% ----------------------------------------------------------------------
\maketitle
%%% ----------------------------------------------------------------------

\section*{Introduction and main result}
In recent years, starting from the very important papers \cite{berest1, norisCPAM}, considerable attention has been paid to the study of entire solutions to the following elliptic system:
\begin{equation}\label{eqEqBU}
\left\{\begin{array}{c}
  -\Delta u+v^2u=0, \\

  -\Delta v+u^2v=0,
\end{array}\right.\ \ u,v>0 \ \ \textrm{in}\ \ \mathbb{R}^N,\ N\geq 1,
\end{equation}
  and to its generalization in the case of arbitrary many components (see for example \cite{berest2,farina,ST,SZentire,wangDG}). As is pointed out in \cite{berest1} this system arises in various fields, even in the study of black holes. However, the interest in the aforementioned references stems from the study of mixtures of repulsive Bose-Einstein condensates in the strong separation limit. The mixture is described by a solution of Gross-Pitaevskii system with strong coupling, and the above system arises in the blow-up limit near the interface \cite{tavaresCVPDE} which separates the segregated components.
More precisely,     essentially  only two components of the strongly coupled Gross-Pitaevskii system are nonzero in the vicinity of the regular part of the interface; to main order their interaction is governed by  a solution of (\ref{eqEqBU}) with linear growth (see \cite{SZ,wang}). It was shown in  \cite{wangDG} that such solutions depend only on one variable (corresponding to the direction orthogonal to the interface). We also refer to \cite{berest1} for an earlier related result in two dimensions  along the lines of the proof of the De Giorgi conjecture for the Allen-Cahn equation in low dimensions.

It was shown in \cite{berest1,berest2} that the ODE version of (\ref{eqEqBU}), that is
\begin{equation}\label{eqEqBUode}
\left\{\begin{array}{c}
  - u''+v^2u=0, \\

  -  v''+u^2v=0,
\end{array}\right.\ \ u,v>0 \ \ \textrm{in}\ \ \mathbb{R},
\end{equation}
  admits precisely a two-parameter family of solutions:
\[
\left(\mu U(\mu x+\xi),\ \mu V(\mu x+\xi)\right),
\]
with scaling parameter $\mu>0$ and translation $\xi \in(-\infty,+\infty)$; for some fixed solution pair $(U,V)$ which  satisfies
the mirror reflection symmetry \begin{equation}\label{eqRefUV}U(-x)\equiv V(x),\end{equation} and has the following asymptotic behaviour at respective infinities:
\begin{equation}\label{eqUder}
U(x)\to 0 \ \textrm{as}\ x\to -\infty;\ \ U'(x)\to 1  \ \textrm{as}\ x\to +\infty.
\end{equation}
We note that the convergence in the previous limits is super-exponentially fast. We also point out that $U$ is strictly increasing and convex. Actually, it was observed in \cite{aftalionSourdis}
that there is a constant asymptotic phase in the asymptotic behaviour of $U$ at $+\infty$. In passing, we would like to mention that a different but rather indirect proof of the uniqueness of $(U,V)$ can be given by combining the results in the latter reference.

Of importance is also the associated linearization of (\ref{eqEqBUode}) about the aforementioned solution $(U,V)$, namely the linear problem
\begin{equation}\label{eqlin1d}
\left\{\begin{array}{c}
  - \phi''+V^2\phi+2UV\psi=0 \\

  - \psi''+U^2\psi+2UV\phi=0
\end{array}\right.\ \ \textrm{in}\ \ \mathbb{R}.
\end{equation}
In particular, the higher order terms in a blow-up analysis of the strong separation limit near the regular part of the interface should be given by linear inhomogeneous problems involving this linearized operator (see \cite{casteras} for the radial case).
It was shown in \cite{berest1}  that the solution $(U,V)$ is linearly non-degenerate, in the sense that the only bounded solutions of the above problem are constant multiples of $(U',V')$ (the element of the kernel coming from the translation invariance of (\ref{eqEqBUode})). Based on this, a solvability
theory for the corresponding inhomogeneous problem was developed in \cite{aftalionSourdis}.

Recalling the discussion leading to (\ref{eqEqBUode}), it is expected that arriving to (\ref{eqlin1d}) rigorously in the strong separation limit should
require showing that bounded solutions of the linearization of the PDE system (\ref{eqEqBU}) about the one-dimensional solution $(U,V)$, namely
\begin{equation}\label{eqLin-}
\left\{\begin{array}{c}
  -\Delta \phi+V^2(x)\phi+2UV(x)\psi=0, \\

  -\Delta \psi+U^2(x)\psi+2UV(x)\phi=0,
\end{array}\right.\ \  (x,y)\in (\mathbb{R},\mathbb{R}^{N-1}),
\end{equation}
depend only on $x$, and thus are constant multiples of $(U',V')$ by the aforementioned result of \cite{berest1}. Conversely, as in related elliptic problems that give rise to interfaces such as the Allen-Cahn or NLS equations (see for example \cite{delannal,delToda,pacard}),  knowledge of this property should be crucial in establishing the persistence of formally constructed approximate solutions for the blow-down Gross-Pitaevskii problem. In fact, this property represents the linearized non-degeneracy of the blow-up profile $(U,V)$ with respect to (\ref{eqEqBU}).
However, it is not clear to us how to adapt the  analogous proofs in the aforementioned references, such as distribution theory and energy methods, for this purpose. Loosely speaking,   in the aforementioned references the outer problem (for $|x|\gg 1$) is $-\Delta w+w=0$, which provides exponential decay of bounded solutions. We stress that this exponential decay property is essential for these methods to apply successfully. In contrast, here the outer problem, say for $x\gg 1$, is  $(-\Delta w,-\Delta z+x^2z)=(0,0)$. It is not even clear  if a bounded  component $w$ should have a decay rate as $x\to +\infty$. Nevertheless,   the linear non-degeneracy of $(U,V)$ when $N=2$ was established in \cite{berest1}, as a consequence of a more general result, in the spirit of the proof of De Giorgi's conjecture for the Allen-Cahn equation in low dimensions. More precisely, the authors considered the system for $\left(\frac{\phi}{U'},\frac{\psi}{V'} \right)$ and derived energy estimates over large balls.

In this paper we will establish this property in any dimension. In other words, we will prove  the following.

\begin{thm}\label{thm}
If  $P,Q \in C^2(\mathbb{R}^N)\cap L^\infty(\mathbb{R}^N)$ solve (\ref{eqLin-}), then they depend only on $x$ and
\[
(P,Q)=a\left(U'(x),V'(x)\right),\ \ (x,y)\in (\mathbb{R},\mathbb{R}^{N-1})
\]
for some constant $a\in \mathbb{R}$.
\end{thm}

Our approach is based on using the maximum principle in unbounded domains, in the spirit of \cite{BCN}, for showing that the solutions $(\partial_{y_i}P,\partial_{y_i}Q)$, $i=1,\cdots,N-1$, of (\ref{eqLin-}) are identically zero. We point out that these solutions converge to zero as $|x|\to \infty$, uniformly in $y\in \mathbb{R}^{N-1}$ (see Lemma \ref{lemUnif}).
The study of rigidity and symmetry properties of solutions to elliptic problems with uniform limits has received considerable attention in recent years, with the maximum principle also playing a central role (see for example \cite{bonheure,farMalchR,farinaSciunzi} and the references therein).
On the other hand, there are  some significant differences with respect to the
previous references. More precisely, the main novelty of our work lies on how to
establish the maximum principle at $x$-infinity.

Let us briefly outline the main steps in the proof of Theorem \ref{thm}.
As in \cite{berest2}, the maximum principle will be applied to the system satisfied by $(\phi,-\psi)$ which is cooperative, where $(\phi,\psi)$ satisfies (\ref{eqLin-}). More precisely:
  $(\phi,\psi)$ satisfies (\ref{eqLin-}) iff $(\phi,-\psi)$ satisfies
  \begin{equation}\label{eqM}
\mathcal{M}\left(
\begin{array}{c}
  \phi \\
  \psi
\end{array}
\right):=
\left(\begin{array}{c}
  -\Delta \phi+V^2\phi-2UV\psi  \\

  -\Delta \psi+U^2\psi-2UV\phi
\end{array}\right) =
\left(\begin{array}{c}
  0  \\

  0
\end{array}\right).
\end{equation}
We will exploit the fact that $(U',-V')$ is a positive solution of $\mathcal{M}=0$.
However, this solution degenerates as $|x|\to \infty$, in the sense that one of its components approaches zero.
We will deal with this degeneracy by splitting $\mathbb{R}^N$ in a large strip $\left\{(x,y)\ :\ |x|\leq L,\ y\in \mathbb{R}^{N-1} \right\}$
  and in the two distant half-spaces with $x\leq -L$ and $x\geq L$, respectively.
In the strip   it holds $(U',-V')\geq (c,c)$ for some $c>0$, which guarantees that $\mathcal{M}$ satisfies the maximum principle there (see also a related discussion in \cite{cabre3}).
On the other hand,  we will show that $\mathcal{M}$ satisfies the maximum in the two distant half-spaces by constructing a   positive super-solution, depending only on $x$, which
 diverges as $x\to \pm \infty$. We point out that to control the difference of the super-solutions with a bounded solution of $\mathcal{M}=0$ as $|y|\to \infty$, we will exploit again   the translation invariance of $\mathcal{M}$ in $y$. To implement the above, we will adopt the viewpoint of \cite{sourdisSerrin} and employ Serrin's sweeping principle \cite{serrin}.

    The rest of the paper is structured as follows: In Section \ref{secmax} we will prove that $\mathcal{M}$ satisfies the maximum principle in the two distant half-spaces. Then, in Section \ref{secmain} we will prove our main result Theorem \ref{thm}.

\section{Maximum principle at infinity}\label{secmax}
In this section we will prove the following.

\begin{prop}\label{proMaxInf}
There exists a large $L>0$ such that the following property holds.
If $\phi_\pm,\ \psi_{\pm}\in C^2\left(\bar{T}_L^\pm\right)\cap L^\infty(T_L^\pm)$   satisfy
\begin{equation}\label{eqMBVP}
\mathcal{M}\left(
\begin{array}{c}
  \phi_\pm \\
  \psi_\pm
\end{array}
\right) =
\left(\begin{array}{c}
  0  \\

  0
\end{array}\right),
\end{equation}
where $T_L^\pm=\left\{(x,y)\ :\ \pm x >L,\ y\in \mathbb{R}^{N-1} \right\}$ and $\mathcal{M}$ is as in (\ref{eqM});
\begin{equation}\label{eqNega}
\phi_\pm,\ \psi_\pm < 0 \ \ \textrm{on}\ \ x=\pm L,
\end{equation}
then
\[
\phi_\pm,\ \psi_\pm < 0 \ \ \textrm{in}\ \ T_L^\pm.
\]
\end{prop}

Clearly, by virtue of (\ref{eqRefUV}), it will be enough to show this for the  $+$ case only.
As may be expected, the proof will rely on the construction of a suitable positive super-solution. The latter will be provided by the pair $\left(Z(x),Z(x)\right)$, where $Z$ is as in the following lemma.

\begin{lem}\label{lemZ}
There exists a large constant $L>0$ and a smooth $Z$ such that
\begin{equation}\label{eqHom}
-Z''-2UVZ=0,\ \ Z>0\ \ \textrm{for}\ \ x>L,
\end{equation}
and
\begin{equation}\label{eqInftyZ}
\lim_{x\to +\infty}Z(x)=+\infty.
\end{equation}
\end{lem}
\begin{proof}
This is a direct consequence of the super-exponential decay to zero of $UV$ as $x\to +\infty$ (keep in mind the comment following (\ref{eqRefUV})). Indeed, this implies that the second order linear ODE in (\ref{eqHom}) admits a solution that diverges  linearly as $x\to +\infty$ (see \cite[Thm. 5.5.1]{sanchez}).
\end{proof}

We can now proceed to the proof of Proposition \ref{proMaxInf}.

\begin{proof}
 As we remarked, it will be enough to consider only the $+$ case. Actually, for notational simplicity, we will drop all the $+$ indexes.

We will adapt Serrin's sweeping principle.
We let
\[
\Theta=\left\{\lambda\geq 0\ :\ \theta Z\geq \phi \ \textrm{and} \ \theta Z\geq \psi \ \textrm{in}\  T_L \ \textrm{for every}\ \theta \geq \lambda \right\},
\]
where $L,\ Z(\cdot)$ are as in Lemma \ref{lemZ}.
Our purpose is to show that $\Theta=[0,\infty)$, which will imply that $\phi,\psi \leq 0$. Then, a simple application of the strong maximum principle in each equation of the system will yield at once that $\phi,\ \psi<0$, as desired. We will show that $\Theta=[0,\infty)$ in the remainder of the proof.

We first observe that $\Theta \neq \emptyset$, that is $\Theta=[\tilde{\lambda},\infty)$ for some $\tilde{\lambda}\geq 0$.
This follows plainly from the assumption that $\phi,\psi$ are bounded in $T_L$, while $Z$ is bounded from below by a positive constant in the same region.

To conclude, we will show that $\tilde{\lambda}=0$. To this end, let us suppose to the contrary that $\tilde{\lambda}>0$.
Then, there would exist sequences $\lambda_n<\tilde{\lambda}$ with $\lambda_n \to \tilde{\lambda}$, $x_n>L$ and $y_n\in \mathbb{R}^{N-1}$ such that
\[
\lambda_n Z(x_n)<\phi(x_n,y_n)\ \ \textrm{or}\ \ \lambda_n Z(x_n)<\psi(x_n,y_n),\ \ n\geq 1.
\]
By passing to a subsequence, if necessary, we may assume that the first inequality in the above relation holds for all $n\geq 1$. Now, thanks to (\ref{eqInftyZ}), the sequence $\{x_n\}$ is bounded and, passing to a further subsequence if needed, we may assume that $x_n \to x_\infty \in [L,\infty)$.

We
consider the sequence of translations:
\[
\Phi_n(x,y)=\phi(x,y+y_n),\ \ \Psi_n(x,y)=\psi(x,y+y_n),\ \ (x,y)\in T_L,\ \ n\geq 1.
\]
The pairs $(\Phi_n,\Psi_n)$ clearly still satisfy (\ref{eqMBVP})-(\ref{eqNega}), are  bounded in $T_L$ uniformly with respect to $n$, and it holds $\lambda_n Z(x_n)<\Phi_n(x_n,0).$  Making use of standard elliptic estimates and a usual compactness-diagonal argument, passing to a further subsequence if necessary, we find that $(\Phi_n,\Psi_n)\to (\Phi_\infty,\Psi_\infty)$ in $C^2_{loc}(\bar{T}_L)$.  The limit $(\Phi_\infty,\Psi_\infty)$ is bounded, continues to satisfy (\ref{eqMBVP}), while
\begin{equation}\label{eqContra0}
\Phi_\infty,\ \Psi_\infty \leq 0 \ \ \textrm{on}\ \ x= L,
\end{equation}
 and
\begin{equation}\label{eqContra1}
\tilde{\lambda} Z(x_\infty)\leq\Phi_\infty(x_\infty,0).
\end{equation}

Since $\tilde{\lambda}\in \Theta$, by definition, we have $\tilde{\lambda} Z\geq \phi$  and $\tilde{\lambda} Z\geq \psi$  in $T_L$, i.e.,
$\tilde{\lambda} Z  \geq \Phi_n$ and  $\tilde{\lambda} Z  \geq \Psi_n$ in $T_L$. In turn, letting $n\to \infty$, we obtain that
\begin{equation}\label{eqContra2}
\tilde{\lambda} Z  \geq \Phi_\infty \ \ \textrm{and}\ \ \tilde{\lambda} Z  \geq \Psi_\infty \ \ \textrm{in}\ \ T_L.
\end{equation}

Recalling (\ref{eqHom}), we note that
\[
-\Delta (\Phi_\infty-\tilde{\lambda}Z)+V^2(\Phi_\infty-\tilde{\lambda}Z)=2UV(\Psi_\infty-\tilde{\lambda}Z)-\tilde{\lambda}V^2 Z<0\ \ \textrm{in} \ T_L.
\]
Hence, in light of (\ref{eqContra1}) and (\ref{eqContra2}), we deduce by the strong maximum principle   that
either $\Phi_\infty \equiv \tilde{\lambda}Z$ or $x_\infty=L$ and $\Phi_\infty (L,0) = \tilde{\lambda}Z(L)$.
The first scenario is easily excluded because $Z$ is unbounded whereas $\Phi_\infty$ is bounded.
On the other hand,  the second scenario is excluded from (\ref{eqContra0}) and the fact that $Z(L)>0$.
We have thus arrived at a contradiction, which completes the proof of the proposition
  \end{proof}

\begin{rem}
If the strict inequality in (\ref{eqNega}) is relaxed to less or equal, then either the same assertion continues to hold or both components are identically zero in $T_L$. This follows readily by applying Hopf's boundary point lemma in the last step of the above proof.
\end{rem}

\section{Proof of the main result}\label{secmain}
In this section we will prove our main result Theorem \ref{thm}. To this end, we will need the following lemma.

\begin{lem}\label{lemUnif}
If $P,\ Q$ are as in Theorem \ref{thm}, then it holds
\[
\nabla P(x,y)\to 0\ \textrm{and} \ \nabla Q(x,y)\to 0 \ \textrm{as}\ |x|\to \infty, \textrm{uniformly in}\ y\in \mathbb{R}^{N-1},
\]
where $\nabla$ applies to both $x$ and $y$.
\end{lem}
\begin{proof}
Since $P,Q$ are bounded, recalling that $U\to +\infty$ in a linear fashion as $x\to +\infty$, $V\to +\infty$ in a linear fashion as $x\to -\infty$, and $UV\to 0$ as $|x|\to +\infty$ super-exponentially fast, by a barrier argument and standard elliptic regularity estimates, we deduce that
\begin{equation}\label{eqPQ}\begin{array}{c}
                  \left|P(x,y) \right|+\left|\nabla P(x,y) \right|\leq Ce^x,\ \ x\leq 0, \ y\in \mathbb{R}^{N-1},\\
                   \\
                  \left|Q(x,y) \right|+\left|\nabla Q(x,y) \right|\leq Ce^{-x},\ \ x\geq 0,\ y\in \mathbb{R}^{N-1},
                \end{array}
\end{equation}
for some constant $C>0$ (see also \cite[Prop. 4.3]{delToda}).

For the remaining directions, we will work as follows. Actually, we will present the argument only for $P$ as that for $Q$ is identical.
Let us suppose, to the contrary, that there exists a constant $c>0$ and a sequence $(x_n,y_n)\in \mathbb{R}\times \mathbb{R}^{N-1}$ such that
$x_n\to +\infty$ and $\left|\nabla P(x_n,y_n) \right|\geq c$.

We consider the translations
\[
P_n(x,y)=P(x+x_n,y+y_n),\ \ Q_n(x,y)=Q(x+x_n,y+y_n).
\]
These satisfy
\begin{equation}\label{eqLin}
\left\{\begin{array}{c}
  -\Delta P_n+V^2(x+x_n)P_n+2UV(x+x_n)Q_n=0, \\

  -\Delta Q_n+U^2(x+x_n)Q_n+2UV(x+x_n)P_n=0,
\end{array}\right.\ \  (x,y)\in (\mathbb{R},\mathbb{R}^{N-1}),\end{equation}
and
\[
\left|\nabla P_n(0,0) \right|\geq c.
\]
Since $P_n,Q_n$ are bounded uniformly with respect to $n$, and $V^2(\cdot+x_n)\to 0,$ $UV(\cdot+x_n)\to 0$ in $C_{loc}(\mathbb{R}^N)$,
by standard elliptic estimates and a usual diagonal-compactness argument, passing to a subsequence if needed, we find that $P_n\to P_\infty$  in $C^1_{loc}(\mathbb{R}^N)$. The limit $P_\infty$ is bounded and harmonic in $\mathbb{R}^N$, while \[
\left|\nabla P_\infty(0,0) \right|\geq c>0.
\]
This, however, contradicts the Liouville theorem.
\end{proof}

We are now ready for the proof of Theorem \ref{thm}.

\begin{proof}
For an arbitrary $i\in \{1,\cdots,N-1\}$, let
\[
\Phi=\partial_{y_i}P(x,y_1,\cdots,y_{N-1}),\ \ \Psi=-\partial_{y_i}Q(x,y_1,\cdots,y_{N-1}).
\]
By the assumption that $P,Q$ are bounded, the exponential decay estimates in (\ref{eqPQ}), and standard elliptic estimates, we deduce that the pair $(\Phi,\Psi)$ is  bounded. Moreover, it clearly satisfies
\[
\mathcal{M}\left(
\begin{array}{c}
  \Phi \\
  \Psi
\end{array}
\right)=
\left(\begin{array}{c}
  0  \\

  0
\end{array}\right),\ \   x\in \mathbb{R},\ \ y\in \mathbb{R}^{N-1}.
\]
 In the sequel, we will use some common notation with the proof of Proposition  \ref{proMaxInf} but hope that no confusion is caused.

 We let
\[
\Theta=\left\{\lambda\geq 0\ :\ \theta U'\geq \Phi \ \textrm{and} \ -\theta V'\geq \Psi \ \textrm{in}\  \mathbb{R}^N \ \textrm{for every}\ \theta \geq \lambda \right\}.
\]
Our purpose is to show that $\Theta=[0,\infty)$, which will imply that $\Phi,\Psi \leq 0$. Then, by same argument applied to the solution $(-\Phi,-\Psi)$ we will  obtain
 that $\Phi, \Psi\equiv 0$, which is the first assertion of the theorem. The other assertion will then follow at once from the one-dimensional non-degeneracy result of \cite{berest1} that we mentioned after (\ref{eqlin1d}). The task of showing that $\Theta=[0,\infty)$ will take up the rest of the proof.

We first observe that $\Theta \neq \emptyset$, that is $\Theta=[\tilde{\lambda},\infty)$ for some $\tilde{\lambda}\geq 0$.
Indeed, since $\Phi,\Psi$ are bounded and $U',-V'$ are positive, there exists a sufficiently large $\bar{\lambda}>0$ such that
\[
\bar{\lambda} U'> \Phi \ \textrm{and} \ -\bar{\lambda} V'> \Psi
\ \textrm{on}\  \bar{S}_L,
\]
where $S_L=\left\{(x,y)\ :\ x\in (-L, L),\ y\in \mathbb{R}^{N-1} \right\}$ with $L>0$ as in Proposition \ref{proMaxInf}.
 Then, Proposition \ref{proMaxInf} (applied with $\phi_\pm=\Phi-\bar{\lambda} U'$ and $\psi_\pm=\Psi+\bar{\lambda} V'$) yields that the above strict ordering continues to hold outside of the strip $S_L$. In other words, we have shown that $\bar{\lambda}\in \Theta$.

To establish that $\tilde{\lambda}=0$, we will argue by contradiction and suppose that $\tilde{\lambda}>0$.
In order to show that the latter is absurd, taking again into account Proposition \ref{proMaxInf}, it suffices to show that there exists a small $\delta>0$ such that
\begin{equation}\label{eqclaim}
(\tilde{\lambda}-\delta) U'> \Phi \ \textrm{and} \ -(\tilde{\lambda}-\delta) V'> \Psi
\ \textrm{on}\  \bar{S}_L.
\end{equation}
If not, we may assume that there  exist  sequences $\lambda_n<\tilde{\lambda}$ with $\lambda_n \to \tilde{\lambda}$, $x_n\in [-L,L]$ with $x_n\to x_\infty$ and $y_n \in \mathbb{R}^{N-1}$ such that
\[
\Phi(x_n,y_n)\geq \lambda_n U'(x_n).
\]
As in the proof of Proposition \ref{proMaxInf}, we consider the translations
\[
\Phi_n(x,y)=\Phi(x,y+y_n),\ \ \Psi_n(x,y)=\Psi(x,y+y_n),\ \ (x,y)\in \mathbb{R}^N.
\]
Up to a subsequence, we find that $(\Phi_n,\Psi_n)\to (\Phi_\infty,\Psi_\infty)$ in $C^2_{loc}(\mathbb{R}^N)$, where  $(\Phi_\infty,\Psi_\infty)$, is bounded and  satisfies the following:
\begin{equation}\label{eqPhi}
\mathcal{M}\left(
\begin{array}{c}
  \Phi_\infty \\
  \Psi_\infty
\end{array}
\right)=
\left(\begin{array}{c}
  0  \\

  0
\end{array}\right),\ \   x\in \mathbb{R},\ \ y\in \mathbb{R}^{N-1},
\end{equation}
\[
\tilde{\lambda} U'\geq \Phi_\infty \ \textrm{and} \ -\tilde{\lambda} V'\geq \Psi_\infty \ \textrm{in}\  \mathbb{R}^N,
\]
\[
\Phi_\infty(x_\infty,0)= \tilde{\lambda} U'(x_\infty).
\]
Since
\[
-\Delta(\Phi_\infty- \tilde{\lambda} U')+V^2(\Phi_\infty- \tilde{\lambda} U')=2UV(\Psi_\infty+ \tilde{\lambda} V')\leq 0\ \ \textrm{in}\ \ \mathbb{R}^N,
\]
we infer by the strong maximum principle that $\Phi_\infty \equiv \tilde{\lambda} U'$. However, recalling (\ref{eqUder}), this is in contradiction with Lemma \ref{lemUnif} which implies that $\Phi_\infty \to 0$ as $|x|\to \infty$.
\end{proof}

\section*{Acknowledgments} The author  would like to express his  thanks to Prof. Terracini for asking him this interesting question while a research fellow at the University of Turin. Part of the paper was written while the author was a visitor in  the University of Ioannina.

\end{document}